\documentclass[letterpaper, 10pt, conference]{ieeeconf}  
\pdfoutput=1
\IEEEoverridecommandlockouts     
\renewcommand{\baselinestretch}{1}

\usepackage{graphics} 
\usepackage{epsfig} 
\usepackage{amsmath} 
\usepackage{multirow}
\usepackage{booktabs}
\usepackage{amsfonts}

\newtheorem{proposition}{\textbf{Proposition}}
\newtheorem{corollary}{\textbf{Corollary}}
\newtheorem{theorem}{\textbf{Theorem}}

\newcounter{remarkcount}
\newenvironment{remark}{\stepcounter{remarkcount}\textbf{Remark} \arabic{remarkcount}}{ \ $\diamond$}

\newcounter{assumpcount}
\newenvironment{assumption}{\refstepcounter{assumpcount}\textbf{Assumption} \arabic{assumpcount}}{ \ $\square$}

\title{\LARGE \bf
Multiple-Model Adaptive Control With Set-Valued Observers
}

\author{Paulo Rosa, Carlos Silvestre, Jeff S. Shamma and Michael Athans 
\thanks{This work was partially supported by Funda{\c c}{\~a}o para a Ci{\^e}ncia e a Tecnologia (ISR/IST pluriannual funding) through the POS\_Conhecimento Program that includes FEDER funds, by the PTDC/EEA-ACR/72853/2006 OBSERVFLY Project, and by the NSF project \#ECS-0501394. The work of P. Rosa was supported by a PhD Student Scholarship, SFRH/BD/30470/2006, from the FCT.}
\thanks{P. Rosa, C. Silvestre and M. Athans are with Institute for Systems and Robotics - Instituto Superior Tecnico,
				Av. Rovisco Pais, 1, 1049-001 Lisboa, Portugal
				{\tt\small prosa@isr.ist.utl.pt, cjs@isr.ist.utl.pt, athans@isr.ist.utl.pt}}
\thanks{M. Athans is also Professor of EECS (emeritus), M.I.T., USA}%
\thanks{J. S. Shamma is with Georgia Institute of Technology, School of Electrical and Computer Engineering, Atlanta, Georgia, United States of America
       {\tt\small shamma@gatech.edu}}%
}

\begin{document}

\maketitle
\thispagestyle{empty}
\pagestyle{empty}

\begin{abstract}
This paper proposes a multiple-model adaptive control methodology, using set-valued observers (MMAC-SVO) for the identification subsystem, that is able to provide robust stability and performance guarantees for the closed-loop, when the plant, which can be open-loop stable or unstable, has significant parametric uncertainty. We illustrate, with an example, how set-valued observers (SVOs) can be used to select regions of uncertainty for the parameters of the plant. We also discuss some of the most problematic computational shortcomings and numerical issues that arise from the use of this kind of robust estimation methods. The behavior of the proposed control algorithm is demonstrated in simulation. 
\end{abstract}

\section{Introduction}

In many realistic applications, the model of a system is only known up to some level of precision, due to uncertain parameters and unmodeled dynamics. 
Sometimes, a \emph{robust non-adaptive controller} is enough to achieve the desired closed-loop performance, e.g., to guarantee a given level of attenuation from the exogenous disturbances inputs to the performance outputs. If, however, the region of uncertainty is large and/or there are stringent performance requirements, such a non-adaptive controller may not exist. To overcome this problem, several solutions are proposed in the literature of adaptive control.

In this paper, we consider an important class of adaptive control architectures, referred to as multiple-model adaptive control (MMAC).
In particular, we are going to address the case where the process model has one parametric uncertainty, $p \in [p_\text{min}, \, p_\text{max}]$. Although several switching MMAC methodologies are available to solve this problem, they all share the same principles: in terms of design, we divide the (large) set of parametric uncertainty, $K$, into $N$ (small) subregions, $K_i$, $i=\{1,\cdots,N\}$ -- see Fig. \ref{fig:region_uncert} -- and synthesize a non-adaptive controller for them; in terms of implementation, we try to identify which region the uncertain parameter, $p$, belongs to, and then use the controller designed for that region. As explained in the sequel, the approach presented herein \emph{discards} the regions where the uncertain parameter, $p$, \emph{cannot} belong.

\begin{figure}[!htbp]
	\centering
		\includegraphics[width=1.75in]{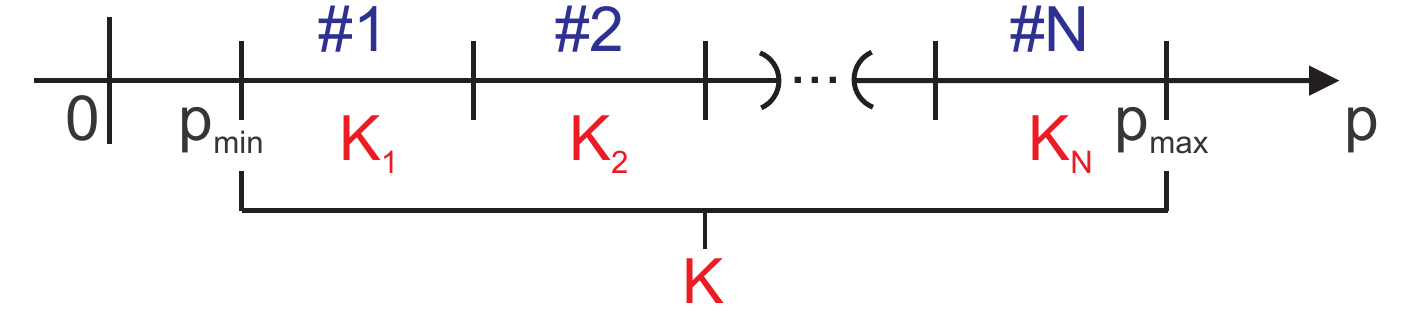}
	\caption{Uncertainty region, $K$, for the parameter $p$, split into $N$ subsets.}
	\label{fig:region_uncert}
\end{figure}

For a list of advantages of this type of control see, for instance, \cite{HespanhaLiberzonMorseAndersonBrinsmeadBruyneApr01}.
Several MMAC architectures have been proposed that provide stability and/or performance guarantees as long as a set of assumptions are met. For instance, \cite{ioannouMMACwM} uses a parameter estimator to select a controller, guaranteeing stability of the closed-loop.
Another MMAC, the so-called robust multiple-model adaptive control (RMMAC), introduced in \cite{sajjadIJAC} and references therein, uses a bank of Kalman filters for the identification system and a hypothesis testing strategy to select the controllers. For this case, although the simulation results indicate that high levels of performance are obtained, the only guarantees that can be provided are in terms of stability -- see \cite{prosaACC09}.
In \cite{shamma09}, calibrated forecasts are used to guarantee the stability of the closed-loop.
The theory of unfalsified control -- see \cite{safonov97} and references therein -- uses the controlled output error to decide whether the selected controller is delivering the desired performance or not.
The authors in \cite{angeli02} use a Lyapunov-based approach to select controllers, and hence require an in-depth knowledge of the plant.
Some of the assumptions required by these methodologies are often unnatural or cannot be verified in practice.

The approach in this paper is somewhat different to the above MMAC architectures. Instead of trying to identify the \emph{correct} region, \emph{i.e.}, the region where the uncertain parameter takes value, by hypothesis testing or parameter estimation, we exclude the \emph{wrong} regions. In other words, if the time-evolution of the inputs and outputs of the plant cannot be explained by a model with uncertain parameter $p$, such that $p \in K_i$, then region $K_i$ cannot be the one which the uncertain parameter belongs to. For dynamic uncertain models, described by differential inclusions, this can be posed as the problem of tracking a differential inclusion -- we remark that there is a rich set of references in the mathematical literature on differential inclusions as discussed in \cite{aubin,smirnov,blanchini1999}.

For linear dynamic models, the problem of ``disqualifying'' regions can be tackled using set-valued observers (SVOs) for linear systems -- see \cite{shammaSVF}. These observers consider that the initial state of the system is uncertain, that there are disturbances acting upon the plant, and that the measurements are corrupted with noise. Therefore, the state estimate, at each sampling time, is a set, instead of a single point.
In this paper, we generalize the observers in \cite{shammaSVF} to uncertain models and provide sufficient conditions for the convergence of the SVOs implemented in a non-ideal environment.

In summary, the approach in this paper is to use SVOs to decide which non-adaptive controllers should \emph{not} be selected. Similarly to other MMAC architectures, we use a bank of observers -- in our case, SVOs --, each of which \emph{tuned} for a pre-specified region of uncertainty. 
However, we utilize the observers to \emph{discard} regions, rather than to identify them. Using this strategy, we are able to provide robust stability and performance guarantees for the closed-loop, even when the model of the plant is uncertain. Moreover, we are able to handle both stable and unstable systems.

This paper is organized as follows: section \ref{sec:svo} illustrates, with an example, how set-valued observers can be used to select regions of uncertainty for the parameters of the plant, while stressing some of the computational shortcomings and numerical issues that arise from the use of this method; section \ref{sec:disc} deals with some of the issues related to the discretization of continuous-time uncertain models and points out some common problems in systems with different sampling times; section \ref{sec:mmacsvo} is devoted to the integration of set-valued observers with a multiple-model control architecture; finally, some general comments on the results obtained are provided in section \ref{sec:conc}.

\section{Set-Valued Observers}
\label{sec:svo}

\subsection*{Set-Valued Observers for Uncertain Plants}

The results presented in this section are a generalization of the ones in \cite{shammaSVF} to models with parametric uncertainties. Due to space limitations, some intermediate calculations are going to be omitted. Consider a plant described by the following discrete-time model:
\begin{equation}
\scriptsize
\left\{
\begin{array}{rcl}
x(k+1) &=& A(k) x(k) + A_\Delta(k)x(k) + L_d (k) d(k) + B(k) u(k)\\
y(k) &=& C(k) x(k) + n(k),
\end{array}\right.
\label{eq:dyn}
\end{equation}
\noindent where $x(0) \in X(0)$, $d(k)$ with $|d(k)|=\displaystyle\max_i |d_i(k)| \leq 1$ are the disturbances, $n(k)$ with $|n(k)|=\displaystyle\max_i |n_i(k)| \leq 1$ is the sensor noise, $u(k)$ is the control input, $y(k)$ is the measured output, $x(k)$ is the state of the system and
$
X(0) := \text{Set}(M_0, m_0),
$
where
$
\text{Set}(M, m) := \left\{ q : M q \leq m \right\}.
$
Furthermore, we assume that
$
A_\Delta(k) = A_1(k) \Delta_1(k) + \ldots+A_{n_A}\Delta_{n_A}(k),
$
\noindent for $|\Delta_i| \leq 1, i=1,\ldots,n_A$. The scalars $\Delta_i$, $i=\{1,\ldots,n_A\}$, represent parametric uncertainties, while the matrices $A_i$, $i=\{1,\ldots,n_A\}$, are the directions which those uncertainties act upon. Therefore, in \eqref{eq:dyn} we are accounting for models with parametric uncertainty, which is critical for adaptive control. For the sake of simplicity, we assume $n_A = 1$, and define $\Delta := \Delta_1$.
Then, $x(k)$ satisfies
\begin{align}
\scriptsize
\scriptstyle
M(k)
\begin{bmatrix}
x(k)\\ x(k-1)\\ z \\d
\end{bmatrix}
\leq
\begin{bmatrix}
B(k)u(k)\\-B(k)u(k)\\ \mathbf{1}\\ \mathbf{1}\\ \tilde{m}(k)\\ m(k-1)
\end{bmatrix} =: m(k),
\label{ineq:Aun}
\end{align}
\noindent for every $z = \Delta x(k-1)$, $|\Delta| \leq 1$, where $\mathbf{1}$ is a column vector of $1$'s with the adequate length, and where
$$
\scriptstyle
M(k) = \scriptsize \begin{bmatrix}
I & -A(k-1) & -A_1(k-1) & -L_d(k-1)\\
-I & A(k-1) & A_1(k-1) & L_d(k-1)\\
0 & 0 & 0 & I\\
0 & 0 & 0 & -I\\
\tilde{M}(k) & 0 & 0 & 0\\
0 & M(k-1) & 0 & 0
\end{bmatrix},
$$
\noindent
$\tilde{M}(k) = \begin{bmatrix}
C(k)\\ -C(k)
\end{bmatrix}$, and
$
\tilde{m}(k) = \begin{bmatrix}
\mathbf{1} + y(k)\\
\mathbf{1} - y(k)
\end{bmatrix}.
$
\noindent Therefore, for each value of $x(k-1)$, \eqref{ineq:Aun} must be verified for every $d$ such that $|d| \leq 1$, and every $z \in Z\left(x(k-1)\right)$, where $Z\left(x(k-1)\right):=\text{co}\left\{ x(k-1), -x(k-1) \right\}$, and $\text{co}\left\{p_1, \ldots, p_m \right\}$ is the smallest convex polytope containing the points $p_1, \ldots, p_m$, also known as convex hull of $p_1, \ldots, p_m$.
Notice that the sensor noise is accounted for in vector $\tilde{m}(k)$.

For each $x(k-1)$, we have $z = \Delta x(k-1)$, for all $\Delta \in \mathbb{R}$ such that $|\Delta| \leq 1$. Notice that
$
z = \Delta x(k-1), \, |\Delta|\leq 1 \Rightarrow |z_i| \leq |x_i(k-1)|
.
$
Thus, the constraints in $z$ can be relaxed to
\begin{equation}
|z_i| \leq |x_i(k-1)|, \, i=1,\ldots,n.
\label{ineq:zi}
\end{equation}

Hence, for $x \in \mathbb{R}^2$, $\Delta \in \mathbb{R}$, we have
\begin{equation}
x(k) \in \displaystyle\bigcup_{i=1,\ldots, 4} \text{Set}(M_i(k), m_i(k)) ,
\label{eq:regunc}
\end{equation}
\noindent where
$
m_i(k) = \left[m(k) \, 0\right]^\text{T}
$
\noindent and
$$
\scriptstyle
M_i(k) = 
\scriptsize
\left[
\begin{array}{cc}
\begin{array}{c}
I\\-I\\ 0\\ 0\\ \tilde{M}(k)\\ 0
\end{array}
&
\begin{array}{cc}
-A(k-1) & -A_1(k-1)\\
A(k-1) & A_1(k-1)\\
0 & 0\\
0 & 0\\
0 & 0\\
M(k-1) & 0
\end{array}
\\
0 & Q_i
\end{array}
\begin{array}{c}
-L_d(k-1)\\L_d(k-1)\\ I\\ -I\\ 0\\ 0\\ 0
\end{array}
\right],
$$
\noindent where
$$
\scriptsize
\begin{array}{l}
Q_1=\begin{bmatrix}-1 & 0 & 1 & 0\\ -1 & 0 & -1 & 0\\ 0 & -1 & 0 & 1\\ 0 & -1 & 0 & -1\end{bmatrix}, 
Q_2=\begin{bmatrix}-1 & 0 & 1 & 0\\ -1 & 0 & -1 & 0\\ 0 & 1 & 0 & -1\\ 0 & 1 & 0 & 1\end{bmatrix},
\\
Q_3=\begin{bmatrix}1 & 0 & -1 & 0\\ 1 & 0 & 1& 0\\ 0 & -1 & 0 & 1\\ 0 & -1 & 0 & -1\end{bmatrix}, 
Q_4=\begin{bmatrix}1 & 0 & -1 & 0\\ 1 & 0 & 1 & 0\\ 0 & 1 & 0 & -1\\ 0 & 1 & 0 & 1\end{bmatrix}.
\end{array}
$$

Although \eqref{eq:regunc}
\noindent does not define, in general, a convex set, it can be overbounded by a convex polytope, represented as $X(k) = \text{Set}(M(k), m(k))$. 

As a final remark, we stress that the computation of $X(k)$, based upon $X(k-1)$, can then be obtained by applying the \emph{Fourier-Motzkin elimination method} (see \cite{shammaSVF,keerthi1987}) to \eqref{eq:regunc}.

\subsection*{Computational and Numerical Issues}
\label{sec:compissues}

This subsection presents a discussion on some of the most important computational and numerical shortcomings of this methodology.

\subsubsection*{Fourier-Motzkin Elimination Method}
The first issue is related to the implementation of the so-called Fourier-Motzkin elimination method, described in \cite{keerthi1987}, that projects polyhedral convex sets on to subspaces. 

The Fourier-Motzkin algorithm leads to a set of linear inequalities, where some of them might be linearly dependent. This may be problematic, since the size of $M(k)$ and $m(k)$ could be increasing very fast with time. To overcome this problem, one has to eliminate the linearly dependent constraints. This can be done by solving several small linear programming problems at each sampling time, making the practical implementation of this type of observers somewhat computationally complex.

\subsubsection*{Union of Convex Polytopes}

We also have to implement the union of several polytopes. This task can be done by noting that, in our case,
$$
\scriptsize
\displaystyle\bigcup_{i=1,\ldots, 2^n} \text{Set}(M, m) \subseteq \text{co}\left\{ \text{Set}(M_1, m_1), \ldots, \text{Set}(M_{2^n}, m_{2^n})  \right\}
.
$$
In one- or two-dimensional spaces, this can be easily implemented. For higher dimensional spaces, one can resort, for instance, to the algorithm in \cite{Verge92anote}.

\subsubsection*{Numerical Approximation of Convex Polytopes}
Another possible shortcoming of the SVOs is related to the numerical approximations used during the computation of the set-valued estimations. In other words, since we do not have infinite precision in the computations that have to be carried out every sampling time to obtain the set-valued estimate $\hat{X}(k)$, the actual set where the state can take value, $X(k)$, need not be entirely contained inside $\hat{X}(k)$ -- see Fig. \ref{fig:seteps}.
Therefore, it may happen that the true state does not belong to $\hat{X}(k)$, and hence we may end up by possibly discarding the region which the parameter actually belongs to.

Thus, a very simple solution is to ``robustify'' the algorithm by slightly enlarging the set $\hat{X}(k)$, as illustrated in Fig. \ref{fig:seteps}. As long as the maximum error in the computation of the set $X(k)$ is known, we have, for every sampling time, $k$, a vector $\epsilon^*(k)$ such that
$
X(k) \subseteq \text{Set}\left(M(k), m(k)+\epsilon^*(k)\right).
$

\begin{figure}[!htbp]
	\centering
		\includegraphics[width=2in]{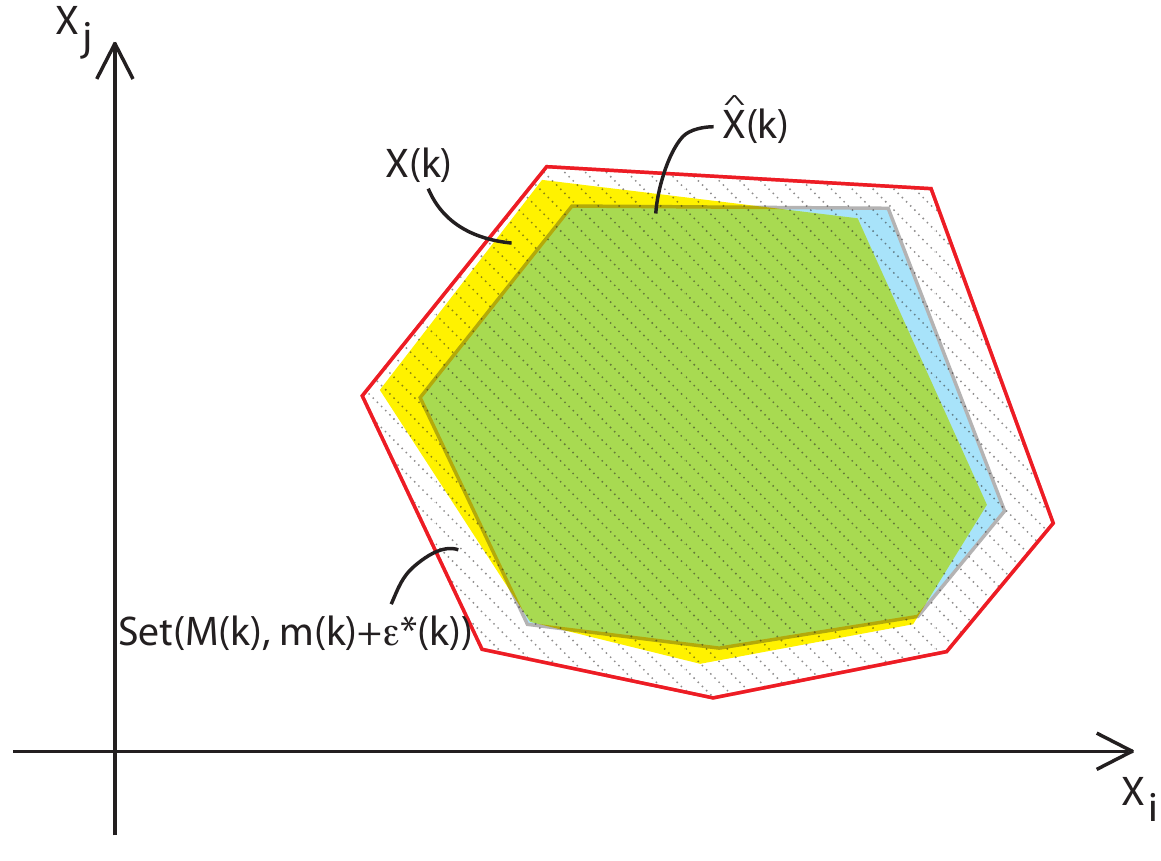}
	\caption{Overbound of set $\hat{X}(k)$ to include $X(k)$.}
	\label{fig:seteps}
\end{figure}

\subsubsection*{Convergence of the Set-Valued Observers}

Using an overbound to guarantee that we do not disqualify the correct estimator also has its shortcomings. Besides adding conservatism to the solution, it may be responsible for the unbounded increase with time of the area of the polytope of the set-valued estimate.

The remainder of this subsection is devoted to the derivation of sufficient conditions that guarantee that, if $X(k)$ is bounded, so does $\hat{X}(k)$.
Consider a plant described by \eqref{eq:dyn} with $A_\Delta (k) = 0$ and $X(0)$ bounded.

The first thing to realize is that the eigenvalues of $A$ must belong to the unit circle, so that we can guarantee that $X(k)$ is bounded. We assume the stability (and, hence, detectability) of the plant throughout the subsection, \emph{i.e.}, the eigenvalues of $A$ are \emph{inside} the unit circle. However, the SVOs can be used with unstable plants.

Let $\Psi(k)$ be the smallest hyper-cube centered at the origin that contains the set $X(k)$, as depicted in Fig. \ref{fig:bounderror}.
Define $\epsilon(k)$ as the maximum distance between a facet of $\Psi(k)$ and the corresponding facet of the estimate $\hat{\Psi}(k)$, as depicted in Fig. \ref{fig:bounderror}.
Next, we try to derive sufficient conditions to guarantee that $\hat{\Psi}(k)$ is bounded. It should be noticed that $\hat{\Psi}(k)$ can be interpreted as a rough approximation of $\hat{X}(k)$, in the sense that $\hat{X}(k) \subseteq \hat{\Psi}(k)$, which means that if $\hat{\Psi}(k)$ is bounded, so does $\hat{X}(k)$.

\begin{figure}[!htbp]
	\centering
		\includegraphics[width=1.5in]{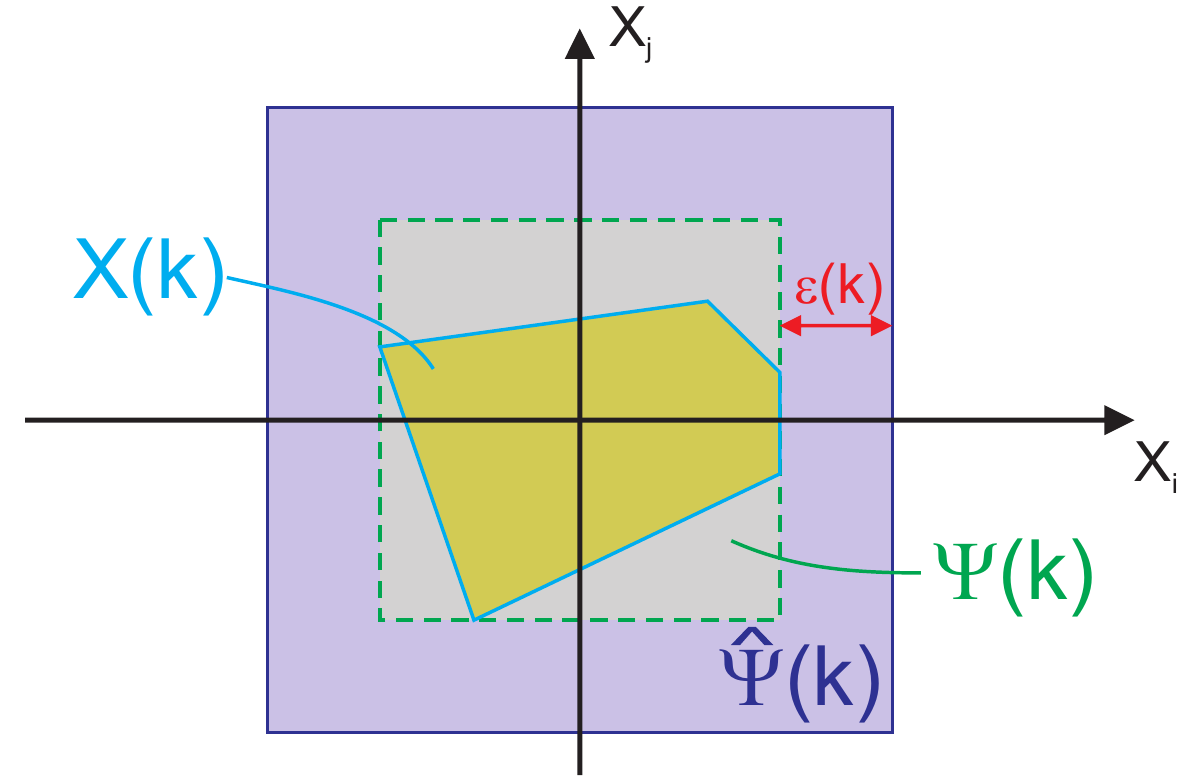}
	\caption{Bounding set $\Psi(k)$, corresponding estimate $\hat{\Psi}(k)$, and maximum numerical error $\epsilon(k)$.}
	\label{fig:bounderror}
\end{figure}

\begin{proposition}
Consider an \emph{asymptotically stable} plant described by \eqref{eq:dyn} with the aforementioned constraints. Suppose that the maximum numerical error (previously defined) at every sampling time is $\epsilon(k)$, with $\epsilon(k) \leq \epsilon^* |x_i(k)|$, for some $0\leq\epsilon^*< \infty$ and  for every $x(k) \in X(k)$. Further suppose that
$
\epsilon^* \leq 1 - \|A\|,
$
\noindent where $\|A\| := \displaystyle\sup_x \frac{|Ax|}{|x|}$.
Then, $\hat{\Psi}(k)$ is bounded.
\end{proposition}

\begin{proof}
Define
$
\delta := \displaystyle\sup_k |Ld(k)|.
$
Notice that $\delta$ is bounded, since $|d| \leq 1$. Then,
$
|x(k+1)| \leq \|A\| |x(k)| + \delta.
$
Since the eigenvalues of $A$ are inside the unit circle, we can find $\gamma > 0$ such that $\|A\| \leq 1-\gamma$.
Hence
\begin{align*}
|x(k+1)| &\leq (1-\gamma) |x(k)| + \delta \\ &= \left[
1-\left(\frac{j-1}{j}\right)\gamma
\right]|x(k)| - \frac{\gamma}{j} |x(k)| + \delta,
\end{align*}
\noindent for any $j \neq 0$.
Thus, for sufficiently large $|x(k)|$ and $j > 0$, we have
$
-\frac{\gamma}{j} |x(k)| + \delta = 0,
$
\noindent which leads to
$
|x(k+1)| \leq \left[1-\left(\frac{j-1}{j}\right)\gamma\right] |x(k)|.
$
However, to overcome the aforementioned numerical issues, we overbound this set by
$
|x(k+1)| \leq \left[1-\left(\frac{j-1}{j}\right)\gamma + \epsilon^* \right] |x(k)|.
$
If $\epsilon^* < \frac{j-1}{j} \gamma$, then $|x(k+1)| < |x(k)|$. Taking the limit as $j$ tends to infinity leads to the desired result.
\end{proof}

\begin{remark}
In an intuitive manner, the above implies that systems that drive their states to zero rapidly can have larger overbounds in the sets $X(k)$ than slower systems.
\end{remark}

\begin{corollary}
Consider a \emph{stable} plant described by \eqref{eq:dyn} with the aforementioned constraints. Suppose that the maximum numerical error at every sampling time is $\epsilon$, with $0 \leq \epsilon < \infty$. Then, $\hat{\Psi}(k)$ is bounded.
\end{corollary}

\begin{proof}
Using a similar approach to that in the previous proposition, we get
$
|x(k+1)| \leq \left[1-\left(\frac{j-1}{j}\right)\gamma\right] |x(k)| + |\epsilon|,
$
\noindent for any $j>0$ and sufficiently large $|x(k)|$. Thus,
$
\left[1-\left(\frac{j-1}{j}\right)\gamma\right] |x(k)| + |\epsilon| \leq \left(1-\rho\right) |x(k)|,
$
\noindent with $\rho > 0$, which concludes the proof.
\end{proof}

\begin{remark}
We stress that this type of errors can be modeled as an exogenous disturbance. This is in agreement with Corollary $1$, since the only requirement is for the system dynamics matrix to be Hurwitz.
\end{remark}

As explained later, these shortcomings of the SVOs do not jeopardize the implementability of the algorithms. In fact, although some of the constraints seem very stringent from a practical point of view, they may not be relevant when used to \emph{discard} regions in a MMAC architecture.
For instance, all the calculations above are only valid for stable plants. Nevertheless, the MMAC architecture with SVOs, introduced in the sequel, can be used to control open-loop unstable plants. This topic will be subject to further discussion.

\subsection*{Mass-Spring-Dashpot Plant}
\label{sec:msd}

We are now going to evaluate the applicability of the SVOs using a simple example.
Consider the mass-spring-dashpot (MSD) plant depicted in Fig. \ref{fig:MSD}.
\begin{figure}[!htbp]
	\centering
		\includegraphics[width=1.25in]{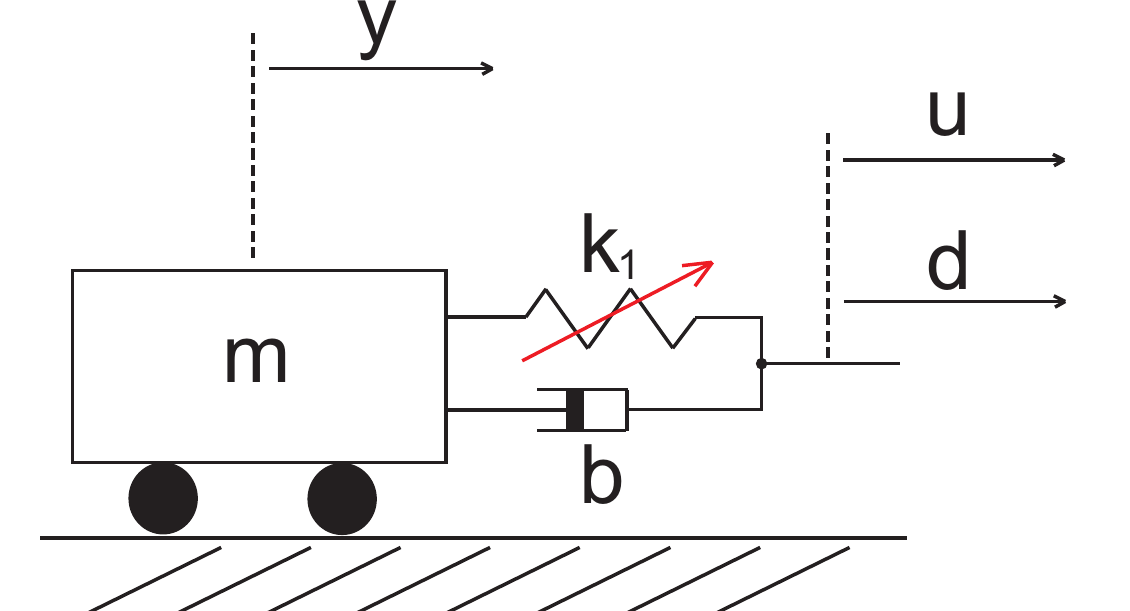}
	\caption{MSD system with uncertain spring constant, $k_1$.}
	\label{fig:MSD}
\end{figure}
The unknown spring constant is
$
k_1 \in [1, \, 5] \text{ N/m}.
$
At this point, we consider $u(\cdot) = 0$. The disturbances and the sensor noise were obtained from uniform random numbers generators, with $|d| \leq 1$, $|n| \leq 0.1$.

\subsection*{Simulations}

We recall that we are trying to derive a multiple-observer system that is able to discard uncertainty regions for the unknown parameters of the plant.
Hence, we start by designing an SVO for the aforementioned MSD plant, with uncertain parameter $k_1$, such that $k_1 \in K_1 := \left[1, \, 5\right]$ N/m. At each sampling time $k T_s$, the SVO produces a set-valued estimate of the state of the plant, by means of a set, $X(k)$. If the actual process is in agreement with the model of the MSD plant, for $k_1 \in K_1$, then the set $X(k)$ can never be empty. If, however, there is a mismatch between the model and the actual plant, caused, for instance, because $k_1 \notin K_1$, then the set $X(k)$ may be empty, for some integer $k \geq 0$.
Thus, we use the SVO -- designed for the uncertain MSD plant with $k_1\in\left[1, \, 5\right] \text{ N/m}$ -- to estimate the state of the plant for different values of $k_1$.

As illustrated in Fig. \ref{fig:test10cmean}, the number of iterations to discard a region may be arbitrarily large, and is mainly dependent upon the disturbances, the sensor noise and the mismatch between the model and the actual plant.

The expected number of iterations for each value of the spring constant, depicted in Fig. \ref{fig:test10cmean}, was obtained by averaging the number of iterations required in each simulation, using different seeds for the random generators used to emulate the disturbances and the sensor noise.
It is clear from the figure that, as one would expect, the closer the true value of the parameter is to the region used to design the observer, the longer it will take before we can discard that region.

\begin{figure}[tbp]
	\centering
		\includegraphics[width=2in]{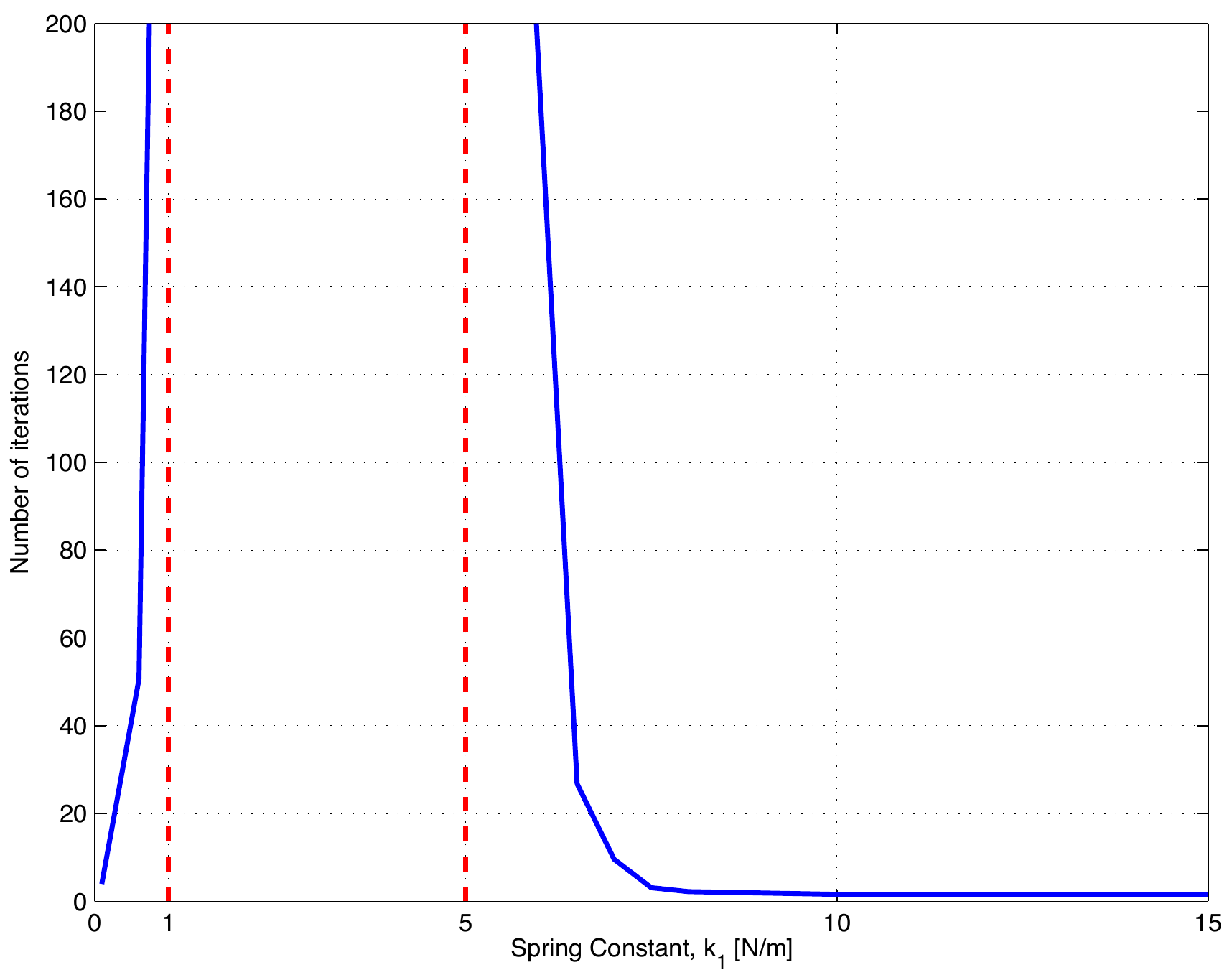}
	\caption{Expected value of the number of iterations required to disqualify region $[1,\,5]$ N/m as function of the spring constant, $k_1$. The dashed red lines bound the valid region for the parameter.}
	\label{fig:test10cmean}
\end{figure}

\section{Discretization Issues}
\label{sec:disc}

This section is devoted to an important topic when dealing with models of realistic systems. Several plants, like the MSD system previously described, are naturally modeled in continuous-time. 
Nonetheless, the aforementioned SVOs can only be used with discrete-time system models.
The usual methods for discretizing LTI systems cannot be readily applied to uncertain models. Therefore, we start this section by deriving discrete-time uncertain models for continuous-time uncertain plants.
However, again due to space limitations, we are going to mention only the most relevant steps.

Another important discretization issue is related to the fact that the SVOs are usually slow in terms of computation, when compared to the dynamics of the plant. In other words, it is required, in general, different sampling times for the control and decision subsystems of the multiple-model adaptive methodology presented in the sequel. Hence, the second part of this section tackles the problem of having different sampling times for the same plant -- the interested reader is also referred to \cite{meyer1975}.

\subsection*{Discretization of Continuous-Time Uncertain Models}

Consider a continuous-time LTI stable plant described by
\begin{equation}
\left\{
\begin{array}{rcl}
\dot{x}(t) &=& A x(t) + B u(t) + L d(t)\\
y(t) &=& C x(t) + n(t).
\end{array}
\right.
\label{eq:dyncont}
\end{equation}

Let $T$ be the sampling time and assume that $u(\cdot)$ and, for the sake of simplicity, $d(\cdot)$, are constant during each sampling interval.
Then, for $t = kT$, we can rewrite \eqref{eq:dyncont} as
\begin{equation*}
\left\{
\begin{array}{rcl}
x((k+1)T) &=& A^d x(kT) + B^d u(kT) + L^d d(kT)\\
y(kT) &=& C x(kT) + n(kT),
\end{array}
\right.
\end{equation*}
\noindent where
$
A^d := e^{AT},\,
B^d := \int_0^T e^{A\tau} B d\tau.
$
In the sequel, we use $k$ instead of $kT$, for the sake of notational simplicity.

Next, we suppose that $A$ is uncertain. Redefine
$
A := \bar{A} + A_1 \Delta, \text{ where } |\Delta| \leq 1.
$
Then, for large $T$,
\begin{equation}
A^d = e^{AT} \approx e^{\bar{A}T}\left( I + A_1 T \Delta + \frac{(A_1 T \Delta)^2}{2} + \cdots \right).
\label{eq:Ad}
\end{equation}

We recall that if $A_1$ is nilpotent, then \eqref{eq:Ad} becomes much simpler. In particular, suppose\footnote{Evaluating how restrictive is this assumption is still a topic of research.} that $A_1^2 = 0$.
Then
$
A^d = e^{AT} \approx e^{\bar{A}T} \left( I + A_1 T \Delta \right) = \bar{A}^d + A_1^d \Delta,
$
\noindent where
$
\bar{A}^d := e^{\bar{A}T},\,
A_1^d := e^{\bar{A}T} A_1 T.
$

For the discretized $B$ (and $L_d$) matrix with uncertainty, we can rewrite it as $B^d = \tilde{B}^d+d_b$, where $d_b$ represents a fictitious disturbance. However, due to lack of space, these calculations are omitted in this paper.

\subsection*{Sampling of Discrete-Time Models}

Due to the computational requirements associated with the set-valued observers (requiring the on-line elimination of several linear inequalities), different sampling times should be used for the control and for the estimation part of the algorithm presented in the next section. Therefore, let $T_s$ be the sampling time for the SVOs, and $T_c << T_s$ be the sampling time for the controllers. Assume that $T_s = m T_c$, where $m$ is a positive integer. We further define
$
A_s = e^{A T_s},
$
and
$
A_c = e^{A T_c}.
$
We consider that the model used for control is described by (omitting the disturbances)
\begin{equation*}
\left\{
\begin{array}{rcl}
x(k+1) &=& A_c x(k) + B_c u(k),\\
y(k) &=& C x(k) + n(k).
\end{array}
\right.
\end{equation*}

Hence,
$
x(k+m) = A_s x(k) + \bar{B} \bar{u}(k),
$
where
\begin{equation*}
\begin{array}{rcl}
\bar{B} &=& \begin{bmatrix}
A_c^{m-1} B_c & A_c^{m-2} B_c &\cdots& A_c B_c & B_c
\end{bmatrix},\\
\bar{u}(k) &=& \begin{bmatrix}
u(k) & u(k-1) & \cdots & u(k+m-1)
\end{bmatrix}^T.
\end{array}
\end{equation*}

Thus, by augmenting the control input, $u(\cdot)$, and using $\bar{n}(k) = n(m k)$, the model used by the SVOs, that is, the model with sampling time $T_s$, can be described by
\begin{equation*}
\left\{
\begin{array}{rcl}
\bar{x}(k+1) &=& A_s \bar{x}(k) + \bar{B} \bar{u}(k),\\
\bar{y}(k) &=& C \bar{x}(k) + \bar{n}(k).
\end{array}
\right.
\end{equation*}

\section{Multiple-Model Adaptive Control with Set-Valued Observers}
\label{sec:mmacsvo}

This section is devoted to the application of SVOs to multiple-model adaptive control (MMAC-SVO) of time-invariant systems.
We start by introducing the architecture of the control scheme and describing the suggested algorithm. Later on, we provide stability and performance guarantees for the closed-loop. A set of simulations illustrating the applicability of the MMAC-SVO is also presented.

\subsection*{Control Architecture}

Figure \ref{fig:MMAC_SVO} depicts the MMAC architecture adopted to use with the SVOs. 
For the sake of simplicity, suppose that the plant depends upon only one uncertain parameter, $k_1$. It is known, however, that $k_1 \in K$, for some set $K \subseteq \mathbb{R}$. The methodology very briefly presented next can be generalized for plants with a higher number of parametric uncertainties.

We follow very closely the method presented in \cite{sajjadIJAC}, to design the MMAC-SVO. For starters, we assume that a single and non-adaptive controller (referred to as global non-adaptive robust controller - GNARC) is not able to achieve the desired performance for the whole uncertainty region. Therefore, we need to divide this region, $K$, into several smaller regions, say $K_1, K_2, \cdots, K_N$, such that $K_1 \bigcup K_2 \bigcup \cdots \bigcup K_N = K$. In order to do so, we first compute the maximum (ideal) performance that we can achieve. This is obviously the case where we know the exact value of the otherwise uncertain parameter, $k_1$. To the controllers designed for fixed values of the uncertain parameter, $k_1$, we call FNARC (fixed non-adaptive robust controller), using the same terminology as in \cite{sajjadIJAC}.

\begin{figure}[!htbp]
	\centering
		\includegraphics[width=0.4\textwidth]{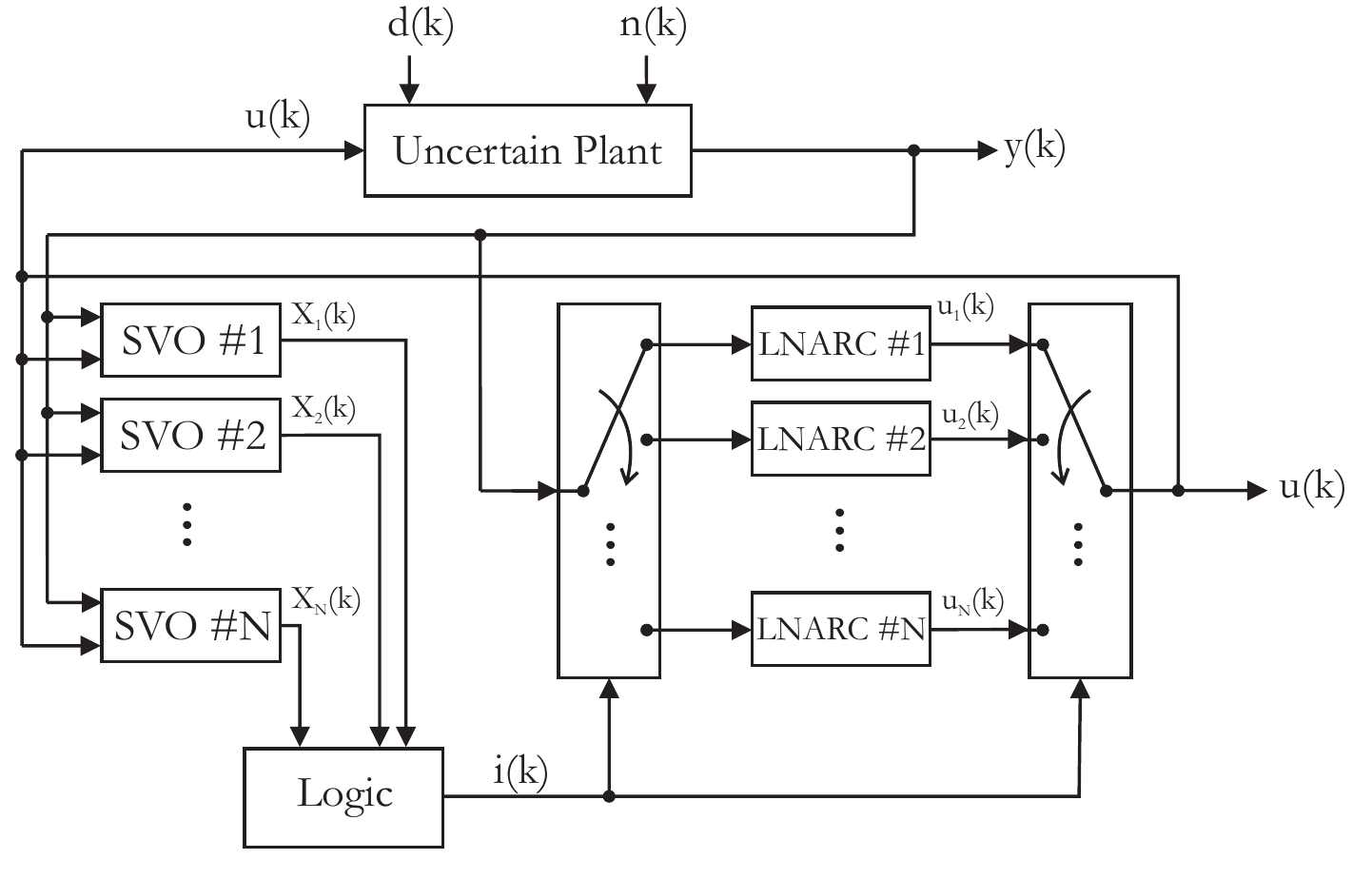}
	\caption{Multiple-model adaptive control with set-valued observers (MMAC-SVO) architecture. $X_i$ is the set estimated by SVO \#i.}
	\label{fig:MMAC_SVO}
\end{figure}

The design proceeds by defining the desired performance for the closed-loop, when the parameter $k_1$ is uncertain. Without loss of generality, we assume that, for each value of the uncertain parameter, $k_1$, we want the performance of the MMAC-SVO not to be smaller than a fixed percentage of the corresponding FNARC. This naturally leads to the splitting of set $K$ into smaller subsets, as previously mentioned. 

For each of these subsets $K_i$, $i=\{1,\cdots,N\}$, a controller referred to as LNARC (local non-adaptive robust controller) is synthesized. We argue that, for any realistic application, these LNARCs should also be robust to plant model error in addition to parameter $k_1$. Thus, using an LFT representation for the system may be useful if, for instance, mixed-$\mu$ controllers are used 
-- see \cite{multivariablefeedback,zhou}.

Furthermore, an SVO should also be designed for each of the subsets, using the methodology previously introduced.
It is important to stress that the SVO for region $K_i$ should be such that, whenever $k_1 \in K_i$, the set $X_i(k)$ is never empty.

\subsection*{Description of the Algorithm}

Having described the architecture and the design procedure of the MMAC-SVO, we propose an algorithm to select the appropriate controller at each sampling time. In reference to Fig. \ref{fig:MMAC_SVO}, this subsection is devoted to the description of the behavior of the block entitled \emph{Logic}. 

Several approaches can be used to tackle this decision problem. In this paper, we suggest a very simple solution that takes into account the fact that, whenever $k_1 \in K_i$, the SVO $\#i$ does never fail, \emph{i.e.}, $X_i(k)$ is never empty. On the other hand, if $k_1 \notin K_i$, then it can happen that, for some $t_0$, we have $X_i(k) = \emptyset$, for all $k \geq t_0$.
Based upon these facts, we suggest the algorithm depicted in Fig. \ref{fig:alg1}.

\begin{figure}[!htbp]
	\centering
		\includegraphics[width=2in]{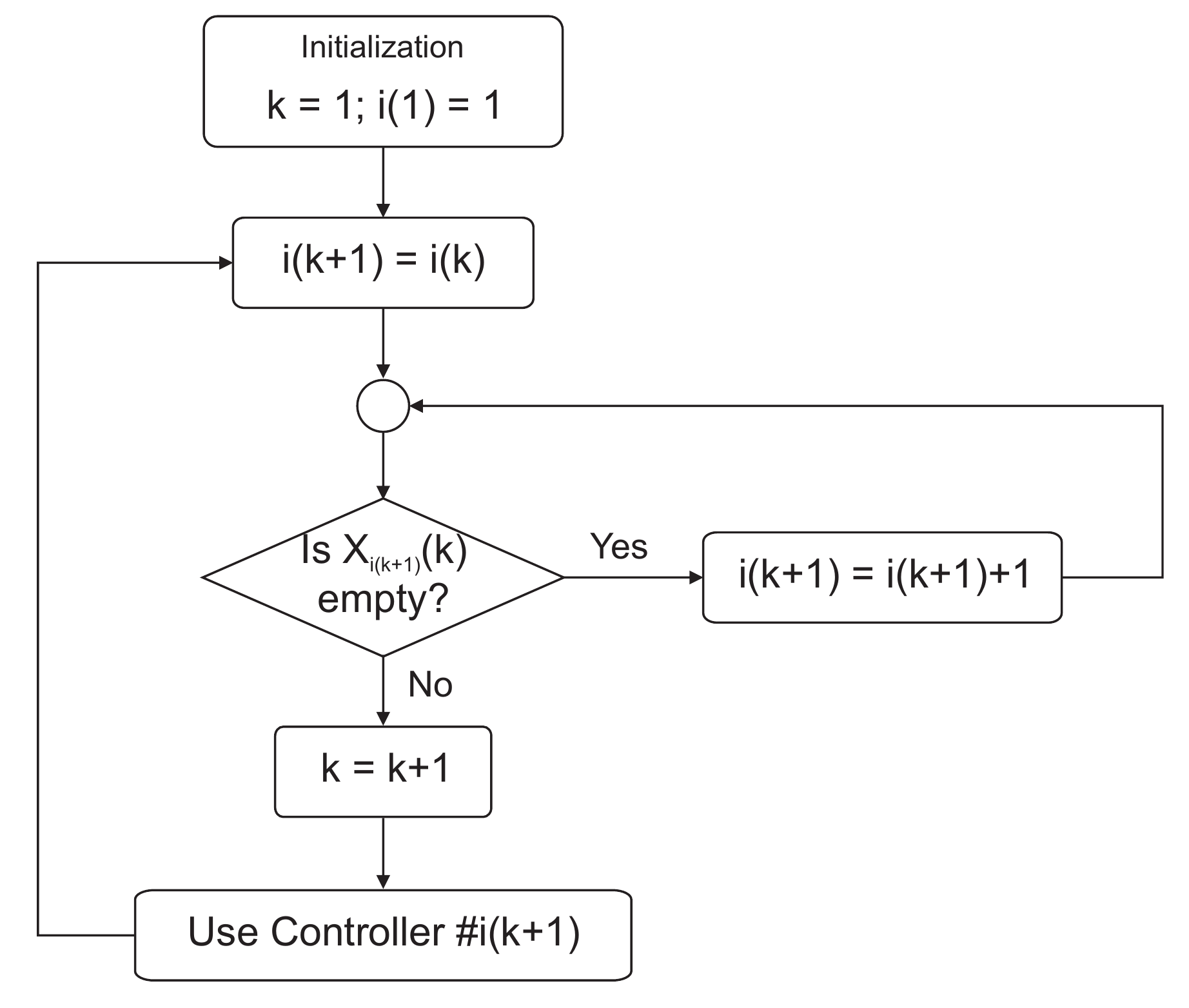}
	\caption{Algorithm proposed to select an appropriate controller at each sampling time.}
	\label{fig:alg1}
\end{figure}

We stress that the main advantage of this algorithm is that we can guarantee robust stability and performance for the closed-loop, as shown in the sequel.

\subsection*{Proof of Stability}

The proof of stability of the closed-loop is based upon the fact that the SVOs are non-conservative, \emph{i.e.}, if $X_i(k) \neq \emptyset$, then the output of the plant, $y(k)$, can be justified by the previous input and outputs, and for some $k_1 \in K_i$. This statement will be explained more formally in the sequel. Once again, for the sake of simplicity, we use a plant with only one uncertain parameter.

Consider a system described by \eqref{eq:dyn} with $n_A = 1$.

Notice that the nominal plant ($\Delta_1(k) = 0$) is a linear time-invariant system.
Further consider a partitioning of the uncertainty set, $K$, as described in the previous subsection ($K = K_1 \bigcup K_2 \bigcup \cdots \bigcup K_N$).
Moreover, we posit the following assumptions.

\begin{assumption}
For each of the uncertainty subsets, $K_i, i=\{1,\cdots,N\}$, there is at least one LNARC, referred to as $C_i(\cdot)$, that is able to stabilize any plant with model \eqref{eq:dyn} and $k_1 \in K_i$.
\label{assump:stabctrl}
\end{assumption}

Let $\tilde{X}(y(k)) = \left\{
x : y(k) = C x + n, |n| \leq 1
\right\},
$ and
$
\bar{X}_i(u(k), k) =\left\{\scriptstyle
x : x = A w + B u(k) +L_d d, w \in X_i(k-1), |d| \leq 1
\right\}.
$

\begin{assumption}
The solution of the SVO $\#i$ is given by
$
X_i(y(k),u(k),k) = \bar{X}_i(u(k),k) \cap \tilde{X}(y(k)).
$
In words, the solutions of the SVOs are non-conservative.
\label{assump:noncons}
\end{assumption}

\begin{assumption}
For any sampling time $k$,
$
x(k) \in X_i(y(k),u(k),k),
$
\noindent for some $i \in \{1,\cdots,N\}$.
\label{assump:svoall}
\end{assumption}

Notice that Assumption \ref{assump:svoall} guarantees that the true plant model belongs to the family of \emph{legal} models of at least one of the SVOs.

\begin{assumption}
The closed-loop system with any of the $N$ controllers does not have a finite escape time.
\label{assump:finitetime}
\end{assumption}

We stress that Assumption \ref{assump:finitetime} is automatically satisfied if all the $N$ controllers are LTI systems -- see \cite{prosaACC09}.

\begin{theorem}
\label{theo:stab}
Suppose Assumptions \ref{assump:stabctrl}$-$\ref{assump:finitetime} are satisfied. Then, using the algorithm previously described, the closed-loop system is input/output stable.
\end{theorem}

\begin{proof}
We first show that the number of switchings is finite. Then, by contradiction, we prove that the closed-loop is input/output stable.

If $X_i(y,u,k) = \emptyset$, then $y(k)$ cannot be explained by the uncertain plant model used by SVO $\#i$. Thus, we switch to a different controller.
According to Assumption \ref{assump:svoall}, at least for one value of $j \in \{1,\cdots,N\}$, $x(k) \in X_j(y,u,k), \forall_k$. Hence, the number of switchings is finite and smaller than $N$. In other words, for some large enough $t_0$, the controller selected at time instant $k \geq t_0$ is always the same.

Next, suppose that $|y(k)|\rightarrow\infty$ as $k\rightarrow\infty$. Let $C_j(\cdot)$ be the controller selected for $k \geq t_0$. 
According to Assumption \ref{assump:noncons}, there is a sequence $(d(k), n(k))$, with $|n| \leq 1$ and $|d| \leq 1$, such that $y(k)$ can be obtained with model \eqref{eq:dyn} with $k_1 \in K_j$.
However, according to Assumption \ref{assump:stabctrl}, controller $C_j(\cdot)$ is able to stabilize any plant with $k_1 \in K_j$. Since $|d|$ and $|n|$ are bounded, and according to Assumption \ref{assump:finitetime}, there cannot exist a sequence $(d(k), n(k))$ such that $|y(k)|\rightarrow\infty$, which is a contradiction.
\end{proof}

\begin{remark}
We stress that the proofs of convergence of the SVOs, presented in section \ref{sec:svo}, are only valid for stable plants. This means that, under suitable conditions, SVO $\#i$ must converge if LNARC $C_i(\cdot)$ is able to stabilize the plant, since in that case the closed-loop system is stable. If, however, controller $C_i(\cdot)$ is not able to stabilize the plant, then SVO $\#i$ may not converge. Nonetheless, this is not a shortcoming, since a non-converging SVO of an unstable plant is always \emph{discarded}, as shown in the proof of Theorem \ref{theo:stab}. Therefore, a LNARC that is not able to stabilize the plant (in the input/output sense) is always \emph{discarded}, as demonstrated above.
\end{remark}

\subsection*{Proof of Performance}
Finally, we provide performance guarantees for the closed-loop system using the MMAC-SVO. For that, we need the following additional assumptions.
We define $\xi(k) = [d(k), \, n(k)]$ as the performance input vector, and $z(k)$ as the performance output vector. In the sequel, we consider that $z(k) \equiv y(k)$, that is, the performance and measurement outputs are the same.

\begin{assumption}
Let $k_1 \in K_i$. Then, controller $C_i(\cdot)$ stabilizes the plant and guarantees a closed-loop $\mathcal{L}_2$-induced norm from the performance inputs to the performance outputs smaller than $\gamma$.
\label{assump:Cperf}
\end{assumption}

\begin{remark}
Assumption \ref{assump:Cperf} ensures the existence of a non-adaptive controller for each value of $k_1 \in K$ that provides the specified performance requirements.
\end{remark}

\begin{assumption}
Let $k_1 \in K_i$. Then

$
\underset{\underset{x(0) \in X(0)}{|n|\leq 1, |d| \leq 1}}{\forall} \quad \underset{U^*,t_2}{\exists} \quad : X_j(t) = \emptyset, \underset{ \underset{u^*(k) \in U^*}{\underset{t \geq t_2}{j \neq i}}}{\forall}.$
\label{assump:ustar}
\end{assumption}

\begin{remark}
Assumption \ref{assump:ustar} can be stated as follows: for every initial state in $X(0)$, there is a set of input sequences, $U^*$, such that, for any $u^*(k) \in U^*$, $k = \{1, 2, \cdots \}$, and for sufficient large $k$, all the SVOs except one (SVO $\#i$) have failed. This assumption can also be seen as a distinguishability condition.
\end{remark}

\begin{theorem}
Suppose Assumptions \ref{assump:noncons}$-$\ref{assump:ustar} are satisfied. Further suppose $u(k) \in U^*$. Then, for some large enough time instant $t_1$, the $\mathcal{L}_2$-induced norm from the performance inputs to the performance outputs, for $k \geq t_1$, is smaller than $\gamma$.
\label{theo:perf}
\end{theorem}

\begin{proof}
According to the proof of Theorem \ref{theo:stab}, there is a time instant $t_0$ such that, for $k \geq t_0$, the selected controller is $C_j(\cdot)$, for some $j \in \{1,\cdots,N\}$. Hence, we only need to prove that this controller guarantees a closed-loop $\mathcal{L}_2$-induced norm from the performance inputs to the performance outputs smaller than $\gamma$. According to Assumption \ref{assump:Cperf}, such controller exists and, if $k_1 \in K_i$, then $C_j(\cdot) \equiv C_i(\cdot)$, \emph{i.e.}, $j = i$.
Thus, if, for a large enough time instant, $k$, all but SVO $\#i$ fail, the proof is complete. This, in turns, is guaranteed by Assumption \ref{assump:ustar}.

\end{proof}

Assumption \ref{assump:ustar} is somewhat unnatural and hard to verify in general. Therefore, we can discard it at the cost of getting a weaker performance index.

\begin{theorem}
Suppose Assumptions \ref{assump:noncons}$-$\ref{assump:Cperf} are satisfied and that the closed-loop system, for fixed controllers, is linear. Further suppose there exists $\bar{\xi}$, such that $|\xi(k)| \leq \bar{\xi}$. Then, we have $|z(k)| \leq \gamma \bar{\xi}$ as $k~\rightarrow~\infty$.
\label{theo:perf2}
\end{theorem}

\begin{remark}
We stress that the performance guarantees provided by Theorem \ref{theo:perf2} are weaker than those of Theorem \ref{theo:perf}. To see this, let $|\xi| \leq \bar{\xi}_1$. However, suppose that we only know an upper bound for $\bar{\xi}_1$, referred to as $\bar{\xi}_2$, such that $\bar{\xi}_2 > \bar{\xi}_1$. Then, if Assumptions  \ref{assump:noncons}$-$\ref{assump:ustar} are satisfied, Theorem \ref{theo:perf} guarantees that $|z| \leq \gamma \bar{\xi}_1$. If Assumptions \ref{assump:noncons}$-$\ref{assump:Cperf} are satisfied, Theorem \ref{theo:perf2} guarantees that $|z| \leq \gamma \bar{\xi}_2 >  \gamma \bar{\xi}_1$ (as $k~\rightarrow~\infty$).
\end{remark}

\begin{proof}
The proof of Theorem \ref{theo:perf2} follows closely the proof for stability and is omitted due to lack of space.
\end{proof}

\subsection*{Simulation: Stable Plant}

We illustrate the applicability of the MMAC-SVO with an example. Recall the MSD plant depicted in Fig. \ref{fig:MSD}. We use a sampling time of $T_c = 1$ ms for the controllers, and a sampling time of $T_s = 500$ ms for the SVOs. The discretization of the plant is done based upon the methodology previously described.
The uncertainty region considered is $K = \left[0.25, \, 9.25\right]$ N/m. 
This (large) region was divided into the following $3$ regions, $K_i$, for $i = \{1,2,3\}$, and, for each region, a mixed-$\mu$ controller that is able to guarantee a certain level of performance was synthesized:
$
K_1 = \left[0.25, \,  5.0\right],\,
K_2 = \left[5.0, \, 7.0 \right],\,
K_3 = \left[7.0, \, 9.25 \right].
$

In the following simulations, we consider that the initial LNARC selected is $C_1(\cdot)$, that is, the mixed-$\mu$ controller designed to guarantee robust performance for the uncertainty region $K_1$. If SVO $\#1$ fails, then we switch to controller $C_2(\cdot)$. If SVO $\#2$ also fails, we switch to controller $C_3(\cdot)$. Notice that at least one of three SVOs cannot fail.

The first simulation was obtained by using $k_1 \in K_1$. Since we start-up with controller $C_1(\cdot)$, there is no need for switching. The result is depicted in Fig. \ref{fig:y_sim1stab}.

\begin{figure}[!htbp]
	\centering
		\includegraphics[width=2.5in]{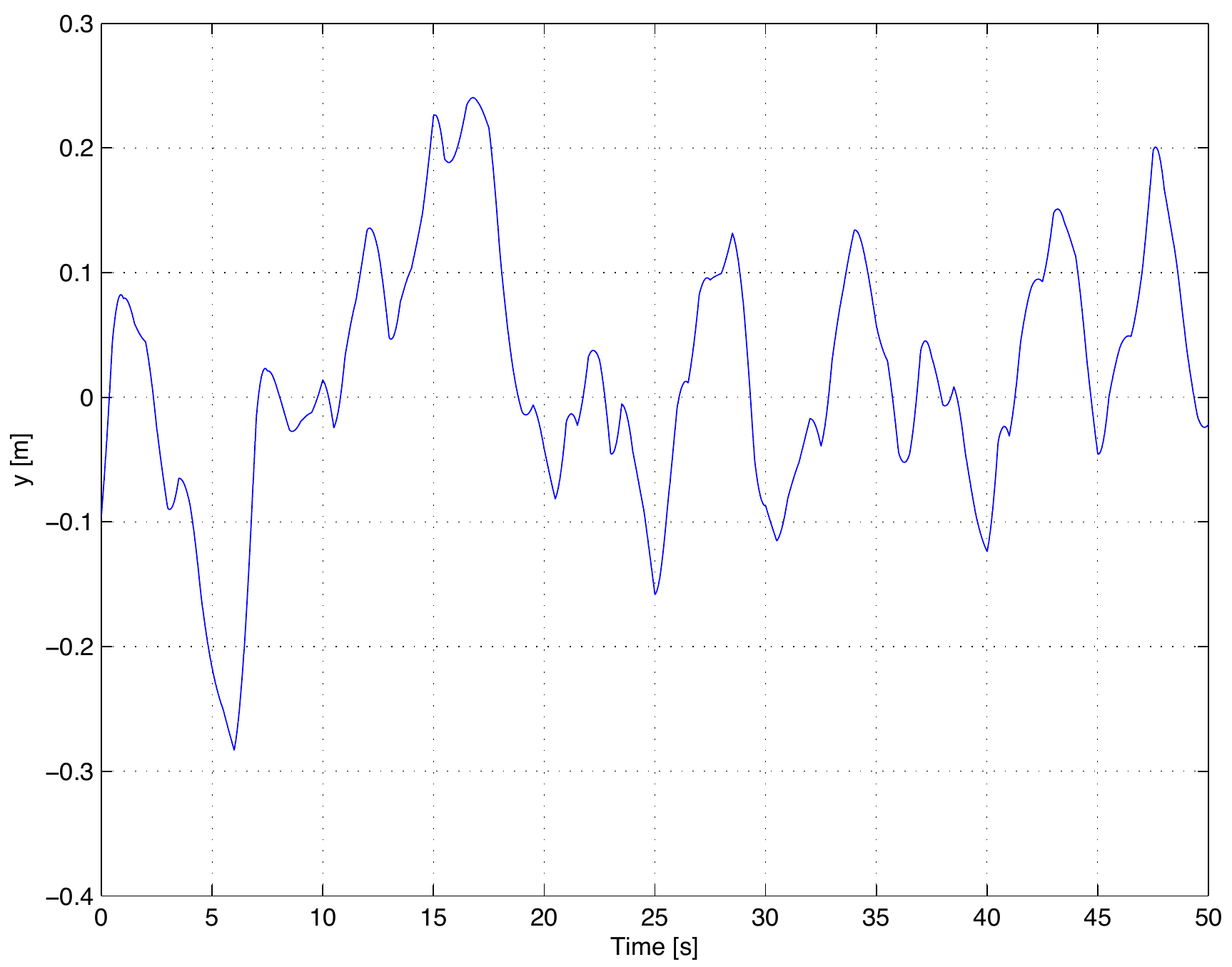}
	\caption{Output of the closed-loop for the MSD-plant with positive $k_1 \in K_1$.}
	\label{fig:y_sim1stab}
\end{figure}

For the second simulation run, we used $k_1 \in K_2$. The initial controller is $C_1(\cdot)$, \emph{i.e.}, we assume that the uncertain plant belongs to $K_1$. However, after two sampling times of the identification scheme, $T_s = 500$ ms, the SVO $\#1$ is not able to explain the measured output, and hence we switch to controller $C_2(\cdot)$. Since the SVO $\#2$ does not fail, we continue using controller $C_2(\cdot)$, as depicted in Fig. \ref{fig:y_sim2stab}.

\begin{figure}[!htbp]
	\centering
		\includegraphics[width=2.5in]{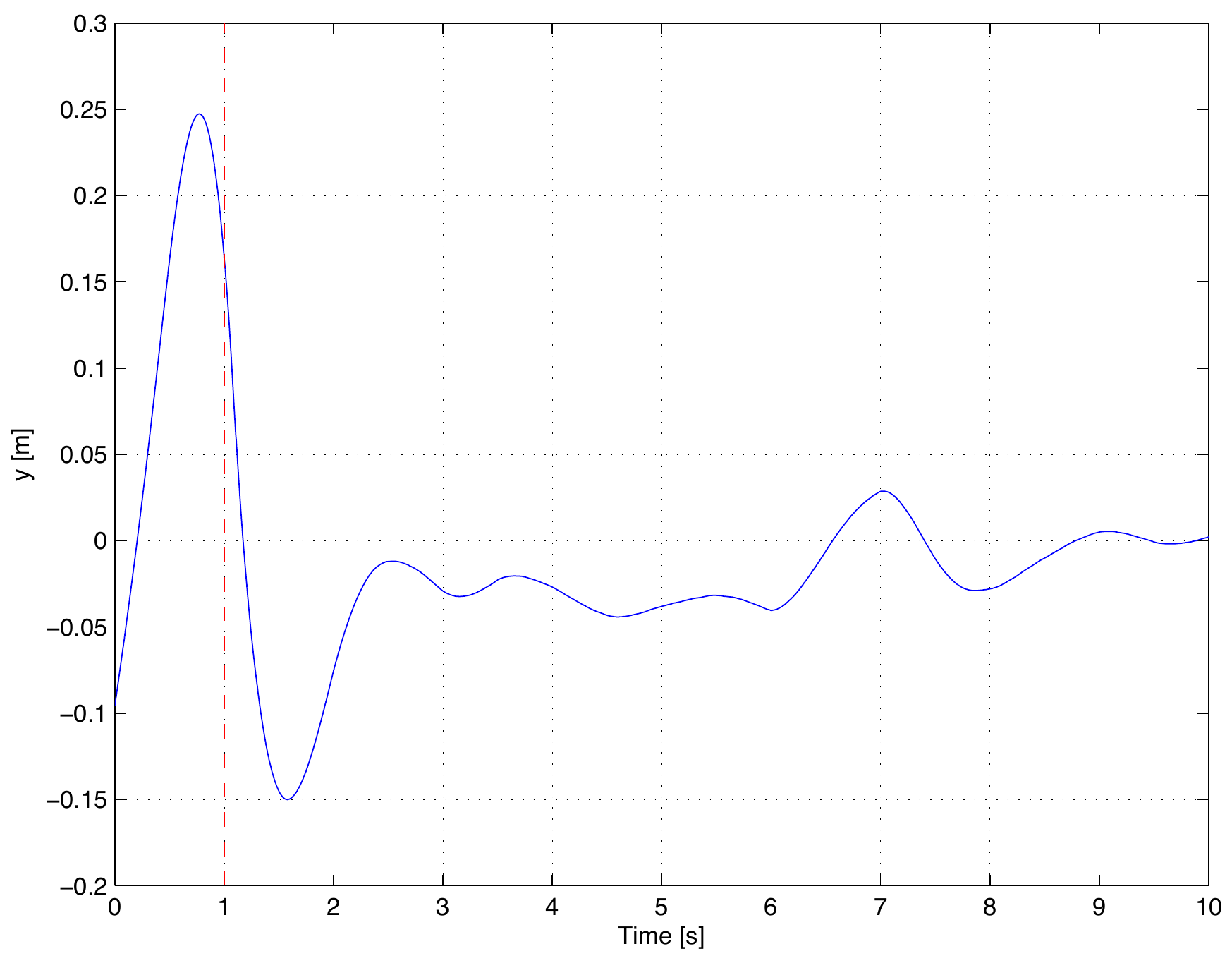}
	\caption{Output of the closed-loop for the MSD-plant with positive $k_1 \in K_2$. The red dashed line indicates the time instant at which the SVO $\#1$ failed, and hence the logic switched to controller $\#2$.}
	\label{fig:y_sim2stab}
\end{figure}

For the last simulation, we used a model with $k_1 \in K_3$. As depicted in Fig. \ref{fig:y_sim3stab}, SVOs $\#1$ and $\#2$ fail at $t = 1$ s and $t = 1.5$ s, respectively. Although $C_2(\cdot)$ is able to stabilize the plant, the algorithm can still decide to switch to controller $C_3(\cdot)$, since the SVO $\#2$ fails.

\begin{figure}[!htbp]
	\centering
		\includegraphics[width=2.5in]{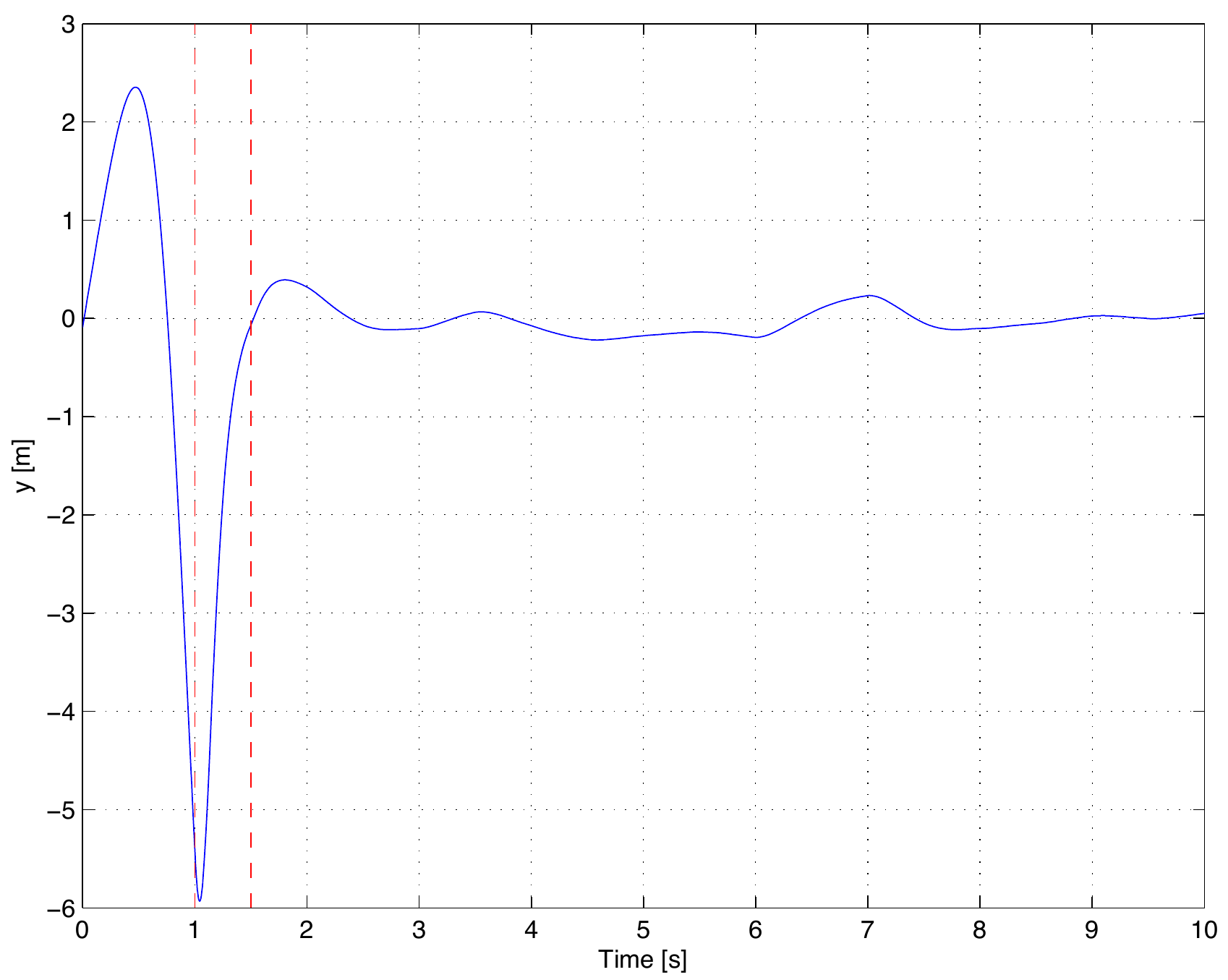}
	\caption{Output of the closed-loop for the MSD-plant with positive $k_1 \in K_3$. The red dashed lines indicate the time instants at which the SVO $\#1$ and $\#2$ failed, and hence the logic switched to controller $\#3$.}
	\label{fig:y_sim3stab}
\end{figure}

It should be noticed that, the smaller the amplitudes of the input and output signals, the harder it is, in general, to discard an uncertainty region, because the models become less distinguishable.
However, we stress that, for the values of the spring constant used, all the controllers were able to stabilize the plant. Nevertheless, the algorithm always picked the controller that provided the performance level we were expecting.

\begin{remark}
For this plant, we have to solve around $10$ to $30$ linear programs, up to $20$ times per iteration, which requires approximately $500$ ms or less in a \emph{Core 2 Duo Pentium}\texttrademark~processor at $2.0$ Ghz.
\end{remark}

\subsection*{Simulation: Unstable Plant}

In this subsection, we make the open-loop plant unstable by using a negative spring stiffness coefficient.

The uncertainty region considered is $K = \left[-42, \, -0.25\right]$ N/m.
This region was divided into the following $3$ regions, $K_i$, for $i = \{1,2,3\}$, and, for each region, a mixed-$\mu$ controller that is able to guarantee a certain level of performance was synthesized:
$
K_1 = \left[-1.65, \,  -0.25\right],\,
K_2 = \left[-13.5, \, -1.65 \right],\,
K_3 = \left[-42, \, -13.5 \right].
$

As in the previous subsection, in the following simulations, we consider that the initial controller selected is $C_1(\cdot)$.
The first simulation was obtained by using $k_1 \in K_1$. Since we start-up with controller $C_1(\cdot)$, there is no need for switching, as depicted in Fig. \ref{fig:y_sim1}.

\begin{figure}[!htbp]
	\centering
		\includegraphics[width=2.5in]{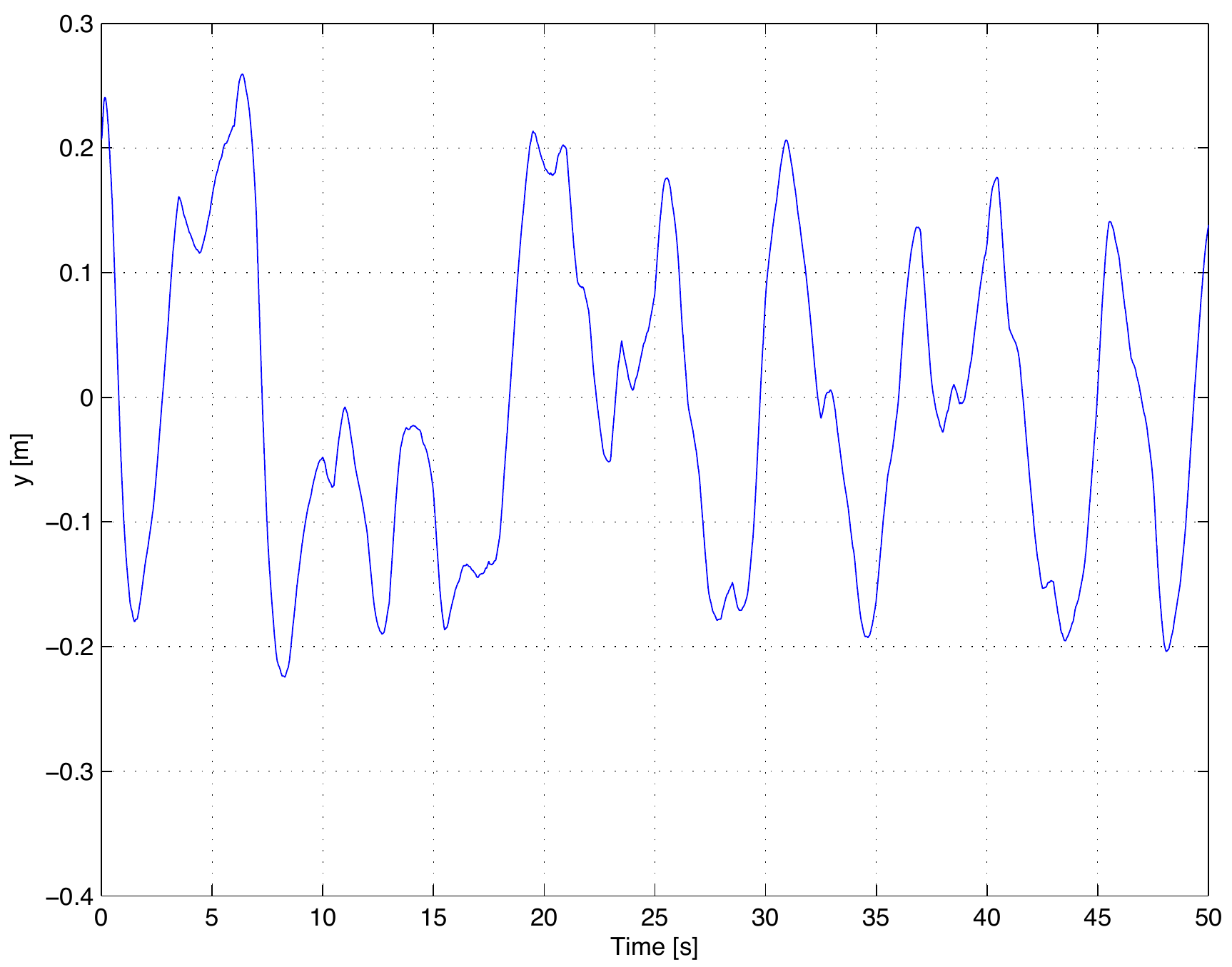}
	\caption{Output of the closed-loop for $k_1 \in K_1$ and open-loop unstable MSD-plant.}
	\label{fig:y_sim1}
\end{figure}

For the second simulation run, we used $k_1 \in K_2$.
The initial controller is $C_1(\cdot)$, \emph{i.e.}, we assume that the uncertain plant belongs to $K_1$. However, after $1$ sampling interval of the identification scheme, $T_s = 500$ ms, SVO $\#1$ is not able to explain the measured output, as shown in Fig. \ref{fig:y_sim2}, and hence we switch to controller $C_2(\cdot)$. Since the SVO $\#2$ does not fail, we continue using controller $C_2(\cdot)$.

\begin{figure}[!htbp]
	\centering
		\includegraphics[width=2.5in]{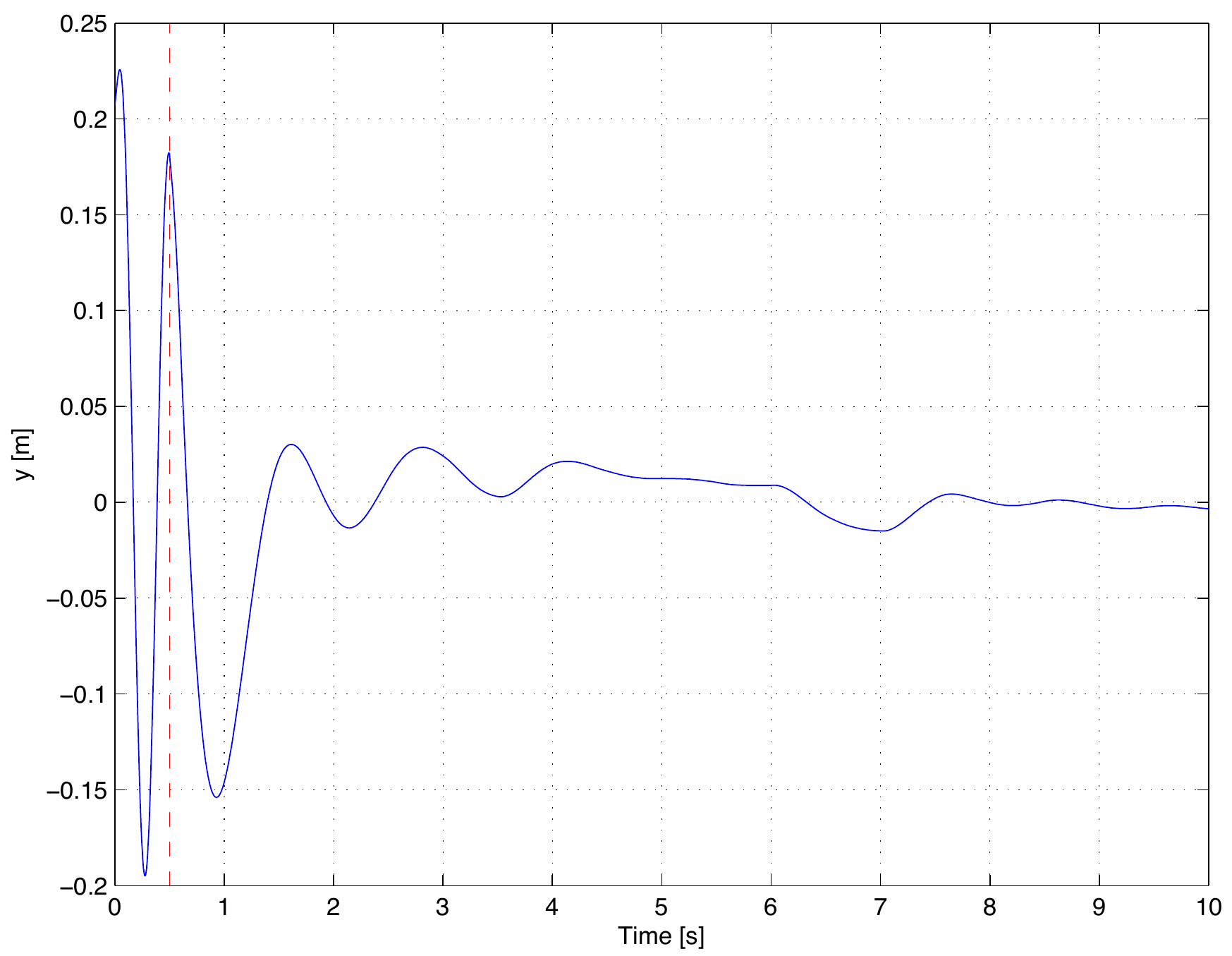}
	\caption{Output of the closed-loop for $k_1 \in K_2$ and open-loop unstable MSD-plant. The red dashed line indicates the time instant at which the SVO $\#1$ failed, and hence the logic switched to controller $\#2$.}
	\label{fig:y_sim2}
\end{figure}

For the last simulation, depicted in Fig. \ref{fig:y_sim3}, we used a model with $k_1 \in K_3$. The SVOs $\#1$ and $\#2$ fail at $t = 500$ ms and $t = 1$ s, respectively. 

\begin{figure}[!htbp]
	\centering
		\includegraphics[width=2.5in]{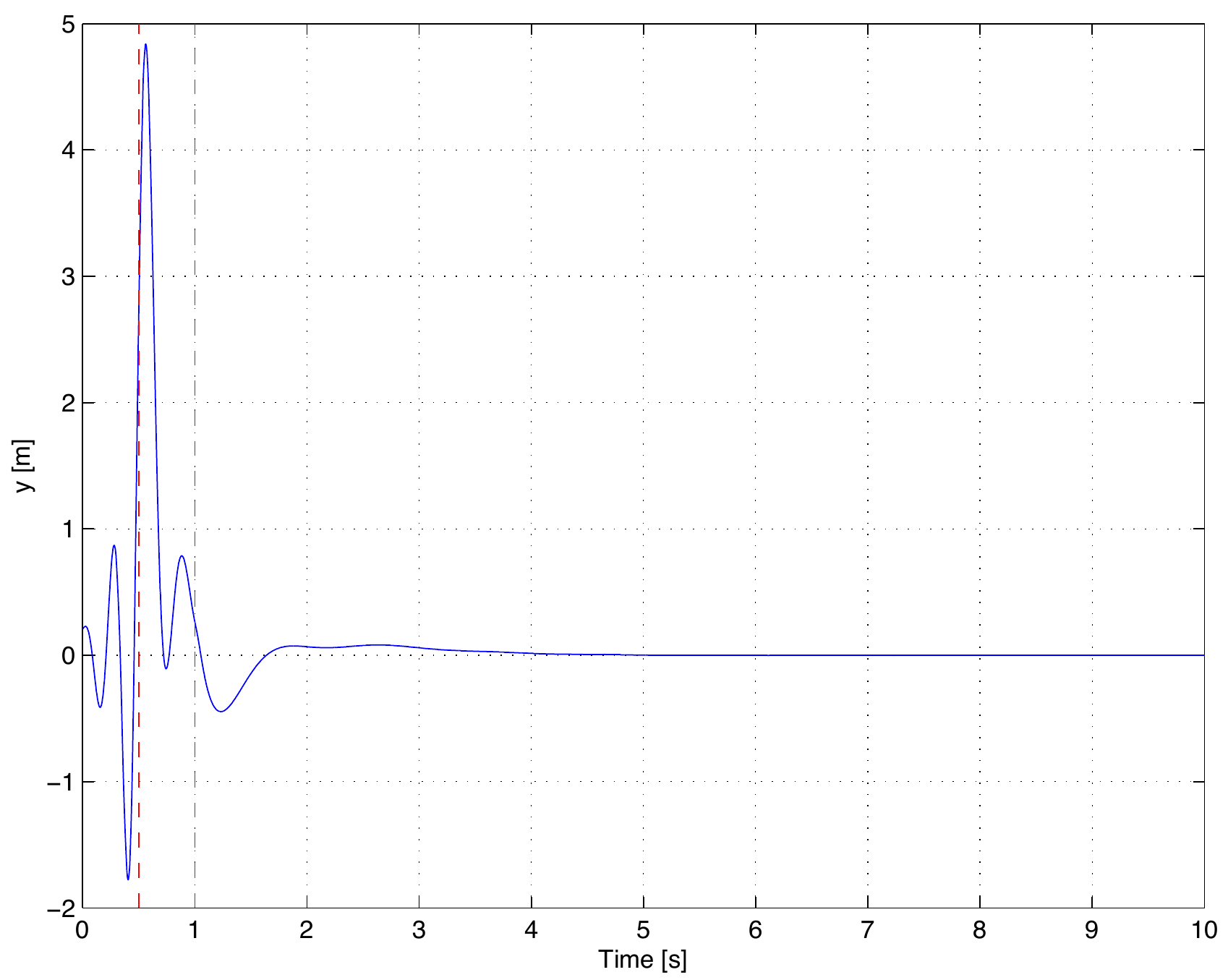}
	\caption{Output of the closed-loop for $k_1 \in K_3$ and open-loop unstable MSD-plant. The red dashed lines indicates the time instants at which the SVO $\#1$ and $\#2$ failed, and hence the logic switched to controller $\#3$.}
	\label{fig:y_sim3}
\end{figure}

It should be noticed that, the smaller the amplitudes of the input and output signals, the harder it is, in general, to discard an uncertainty region, because the differences between models become less visible.

\section{Conclusions}
\label{sec:conc}

This paper illustrated how to use set-valued observers (SVOs) in a multiple-model adaptive control (MMAC) architecture, providing robust stability and performance guarantees for plants with uncertain models. The solution presented can be obtained in a straightforward manner from the discrete-time state-space representation of the plant model. This MMAC architecture is able to handle both stable and unstable uncertain systems, as illustrated with the simulation of a mass-spring-dashpot (MSD) plant.

In this paper, the applicability of the SVOs was extended to plants with parametric uncertainty. We further overcame some numerical issues related to the implementation of the SVOs, by making the algorithms involved more robust to numerical errors.

As a shortcoming of this approach, the computational burden associated with the implementation of the SVOs is highlighted.

\section{ACKNOWLEDGMENTS}

We wish to thank our colleagues Antonio Pascoal, Pedro Aguiar, Vahid Hassani and Jos{\'e} Vasconcelos for the many discussions on the field of robust adaptive control.

\bibliographystyle{IEEEtran}
\bibliography{references}

\end{document}